\documentclass[12pt,a4paper]{amsart}

\usepackage[utf8]{inputenc}
\usepackage[T1]{fontenc}
\usepackage[english]{babel}

\usepackage{indentfirst}
\usepackage{amssymb}
\usepackage{amsfonts}
\usepackage{amsmath}
\usepackage{amsthm}
\usepackage[top=2.5cm, bottom=2.5cm, left=2.5cm, right=2.5cm]{geometry}
\usepackage{amsopn}
\usepackage[colorlinks]{hyperref}
\usepackage{float}
\usepackage{graphicx}
\usepackage{color}
\usepackage{xspace,colortbl}
\usepackage{xfrac}
\usepackage{siunitx}
\usepackage{bbm}
\usepackage{enumitem}
\usepackage[dvipsnames]{xcolor}
\usepackage{array}
\usepackage{booktabs}
\usepackage{multicol}
\usepackage{multirow}
\usepackage{xspace}

\newtheorem{theorem}{Theorem}[section]

\newtheorem{lemma}[theorem]{Lemma}

\theoremstyle{definition}

\newtheorem{remark}[theorem]{Remark}

\newcommand{\R}{\mathbb{R}}

\newcommand{\qziqdi}{$qZI$ and $qDI$\xspace}

\DeclareMathOperator*{\argmin}{arg\,min}

\DeclareMathOperator{\indicator}{\mathbbm{1}}

\DeclareSymbolFont{bbsymbol}{U}{bbold}{m}{n}
\DeclareMathSymbol{\ind}{\mathbin}{bbsymbol}{'061}

\title[Estimation of quantile inequality curves and measures...]{Estimation of quantile inequality curves and measures based on grouped data}
\author[A{.} Jokiel-Rokita]{Alicja Jokiel-Rokita}
\author[S{.} Pi\c{a}tek]{Sylwester Pi\c{a}tek}


\address[A.J.-R. and S.P.]{Faculty of Pure and Applied Mathematics\\ Wroc{\l}aw University
	of Science and Technology\\
	Wybrze\.ze Wyspia\'nskiego 27,
	50-370 Wroc{\l}aw, Poland
}

\address[S.P.]{Dioscuri Centre in Topological Data Analysis \\
Institute of Mathematics\\
Polish Academy of Sciences\\
	\'{S}niadeckich 8,
	00-656 Warsaw, Poland
}

\email{alicja.jokiel-rokita@pwr.edu.pl}

\email{sylwester.piatek@pwr.edu.pl}

\begin{document}

\begin{abstract}
Estimation of quantile inequality curves and measures is considered in 
a~parametric model based on grouped data.
The unknown parameters of the distribution are estimated using the minimum divergence method, using various $\phi$-divergences. 
The consistency of the plug-in estimators of the inequality curves and measures and the asymptotic normality of the indices estimators are proved. 
In a~simulation study, the methods are verified and compared in terms of the accuracy of the estimation. The practical applications of the proposed methods
are illustrated by the analysis of two real data sets.
\end{abstract}

\maketitle

\section{Introduction}\label{sec:intro}

Income inequality remains an important topic of research. 
In many cases, inference must be done based on limited information when data are available only in grouped form. 
This type of data arises in various situations, 
for example, if data are aggregated to avoid sharing too much information or as responses to a~survey question with a~limited number of options.
In this paper, we focus on the~case where observations are given in the form of a~frequency table with fixed limits of the bins.
Our goal is to estimate the quantile inequality curves and measures in a~parametric model based on such data. 

The most popular and widely used tools for income inequality analysis are the Lorenz curve and the Gini index (see, for example, \cite{Mukhopadhyay2021}).
However, the Lorenz curve and the Gini index, corresponding to a~nonnegative random variable, are well-defined only when the expected value of this variable is finite. 
But in the case of income data analysis, where inequality curves and measures are of particular importance, 
the Pareto or Dagum distributions are often adopted, and for certain parameter values these do not have finite expected values
Some alternative curves and measures which are well-defined also for probability distributions having infinite expected values can be considered instead. 
These can be quantile versions of Lorenz curves (see, for example, \cite{Prendergast2016} or \cite{Siedlaczek2018}) and quantile counterparts of Zenga and $D$ curves (see, for example, \cite{Prendergast2018} and \cite{JokPia2024}), namely $qZ$ and $qD$ curves, together with their associated inequality measures $qZI$ and $qDI$.
We focus here on two quantile inequality curves $qZ$ and $qD$ and measures $qZI$ and $qDI$ proposed by Prendergast and Staudte \cite{Prendergast2018}. 

Let $Q_{\theta}$, $\theta\in\Theta$, be a~quantile function of an observable nonnegative random variable.
The curves $qZ$ and $qD$, corresponding to the distribution with quantile function $Q_{\theta}$, are defined as follows
\begin{equation}\label{e:qZ}
qZ(p;Q_{\theta})=1-\frac{Q_{\theta}(p/2)}{Q_{\theta}((1+p)/2)}, \qquad p\in(0,1),
\end{equation}
\begin{equation}\label{e:qD}
qD(p;Q_{\theta})=1-\frac{Q_{\theta}(p/2)}{Q_{\theta}(1-p/2)}, \qquad p\in(0,1),
\end{equation}
and $qZ(0;Q_{\theta})=qZ(1;Q_{\theta})=1,$ $qD(0;Q_{\theta})=1,$ $qD(1;Q_{\theta})=0.$
The corresponding indices $qZI$ and $qDI$ are defined as the area under the inequality curves $qZ$ and $qD$, respectively, namely,
\begin{equation*}\label{e:qZI}
qZI(Q_{\theta})=\int_{0}^{1}qZ(p;Q_{\theta})\,dp
\end{equation*}
and
\begin{equation*}\label{e:qDI}
qDI(Q_{\theta})=\int_{0}^{1}qD(p;Q_{\theta})\,dp.
\end{equation*}
Figure \ref{fig:qZ_qD_Weibull_example} shows examples of the $qZ$ and $qD$ curves corresponding to the Weibull distribution with various values of the shape parameter.
Both curves are scale invariant, thus the scale parameter does not affect the values of the curves.

\begin{figure}
    \centering
    \includegraphics[width=1.0\linewidth]{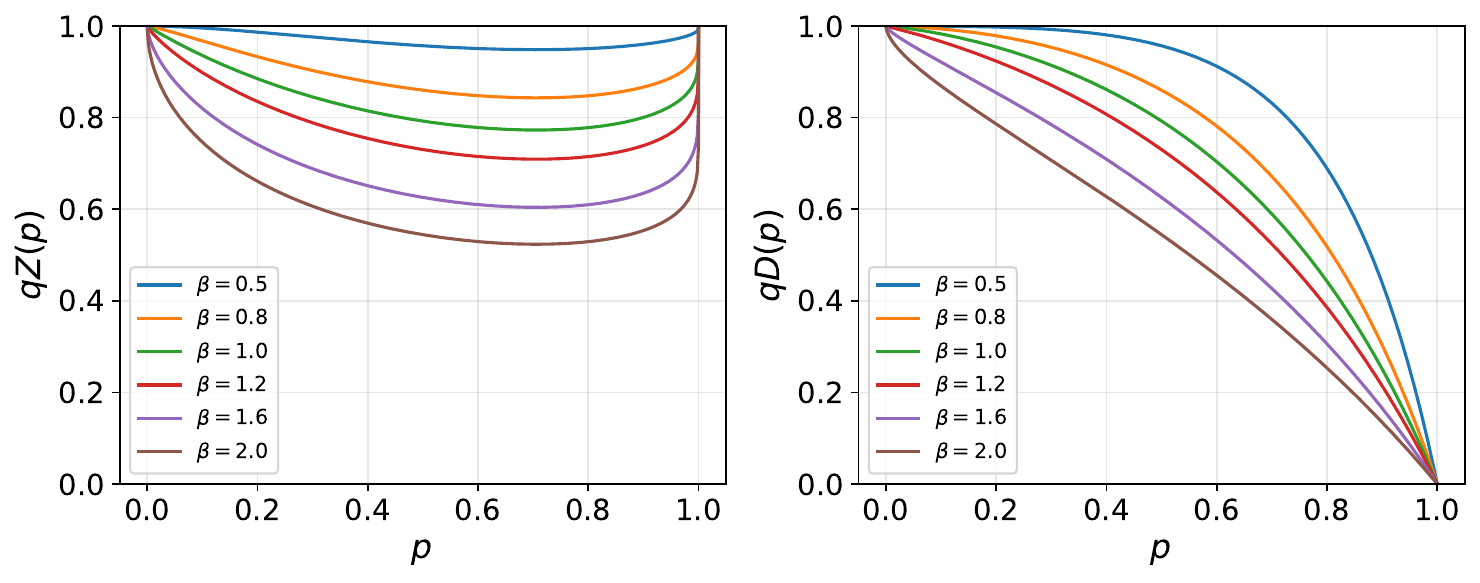}
    \caption{Example of $qZ$ and $qD$ curves for random variables from the Weibull distribution with shape parameters varying between 0.5 and 2}
    \label{fig:qZ_qD_Weibull_example}
\end{figure}

Estimation of the $qZ$ and $qD$ curves and the $qZI$ and $qDI$ indices in a~parametric model based on the complete sample was considered by Piątek \cite{Piatek2025}.
In this paper, we focus on estimating these curves and measures based on grouped data.
We propose using plug-in estimators $\widehat{qZ}(\cdot;Q_{\hat{\theta}})$,
$\widehat{qD}(\cdot;Q_{\hat{\theta}})$, $\widehat{qZI}(Q_{\hat{\theta}})$, and $\widehat{qDI}(Q_{\hat{\theta}})$, where $\hat{\theta}$ is a~minimum divergence estimator of the unknown parameter $\theta.$
We investigate the properties of the proposed estimators of the inequality curves and measures and compare their accuracy when estimators of the parameter $\theta$ are determined based on various divergences.

The estimation of parameters based on grouped data was considered in the literature, for example, in 
\cite{VictoriaFeser1997}, \cite{Johnk2002}, \cite{Spasova2024}.
Dedduwakumara and Prendergast \cite{Dedduwakumara2019} estimated the Gini index and several other inequality measures using linear interpolation to reconstruct the densities from bins when $Q_{\theta}$ corresponds to a~generalised lambda distribution.
The problem of estimating the inequality measures (mainly in the context of estimating the Gini index) based on grouped data was also considered in \cite{Jorda2021}, \cite{Miljkovic2021}, \cite{Lyon2016}, \cite{Carr2022}.

The article is organised as follows. 
Section \ref{sec:methodology} describes the model and methods for estimating both the model parameters and the corresponding curves and inequality measures.
The properties of the estimators considered are proved in Section \ref{sec:properties}.
The results of the simulation study are presented in Section \ref{sec:simulation_study}. 
Two examples of analysis of real data sets are included in Section \ref{sec:real_data}. Finally, the results are concluded in Section \ref{sec:conclusions}.

\section{Minimum divergence estimators of the parameters and the plug-in estimators of curves and indices}\label{sec:methodology}

\subsection{Model assumptions}\label{sec:model}

Let $X_1,\ldots,X_n$ be iid random variables from a~distribution with the cumulative distribution function $F_{\theta_0}$ being an element of the family $\{F_\theta\colon\theta\in\Theta\}$.
We make the following assumptions regarding the distribution function.
\begin{enumerate}[label={(A}\arabic*{)}]
    \item For every $\theta\in\Theta$, the cdf $F_\theta$ is an absolutely continuous and strictly increasing function, $F_{\theta}(0)=0$;
    \item the true value $\theta_0$ of the $s$-dimensional parameter $\theta$ is an interior point of a~parameter space $\Theta$, which is a~compact subset of $\mathbb{R}^s$; \label{ass:interior}
    \item for every $p\in(0,1),$  $F_{\theta}^{-1}(p)=Q_{\theta}(p)$ is a~continuous function of the parameter $\theta.$
\end{enumerate}
Let $[0,c_1), [c_1,c_2), \ldots, [c_{k-1},\infty)$ be a~partition of the support of the random variables $X_i$ into $k$ disjoint intervals. We assume that $k>s+1$, to guarantee a~positive number of the degrees of freedom in the goodness-of-fit statistics.
For each $i=1,\ldots,n,$ we have no information about the value of the variable $X_i$, but we only know which interval this value belongs to.
Denote by $g$ a~vector function of the parameter $\theta$ such that $g(\theta)=(g_1(\theta),\ldots,g_k(\theta))$, where 
\begin{align*}
    g_i(\theta)=F_\theta(c_i)-F_\theta(c_{i-1})
\end{align*}
with $c_0=0$ and $c_k=\infty$.
By $\Pi_0$ we denote a~set of probability vectors such that $g$ maps each element $\theta\in \Theta$ to an element from $\Pi_0 \subset \Delta_k = \{P = (p_1, p_2, \ldots, p_k)\colon
p_i \geq 0;\ i = 1, \ldots, k \text{ and } \sum_{i=1}^{k} p_i = 1\}$.
We also make the following assumptions regarding the function~$g.$
\begin{enumerate}[label={(B}\arabic*{)}]
    \item $g_{0i} \equiv g_i({\theta}_0) > 0$ for $i = 1, \ldots, k$. Hence ${g_0}$ is an interior point of the $(k-1)$-dimensional simplex $\Delta_k$;
    \item the mapping $g\colon \Theta \to \Delta_k$ is totally differentiable at ${\theta}_0$, so that the partial derivatives of $g_i$ with respect to each $\theta_j$ exist at ${\theta}_0$ and $g({\theta})$ has a~linear approximation at ${\theta}_0$ given by
    \begin{equation*}
  g(\theta)
  = g(\theta_{0})
  + \frac{\partial g(\theta_{0})}{\partial\theta}\,(\theta - \theta_{0})
  + o\bigl(\|\theta - \theta_{0}\|\bigr),
  \qquad \text{as } \theta \to \theta_{0},
    \end{equation*}
    where $\frac{\partial g({\theta}_0)}{\partial {\theta}}$ is a~$k \times s$ matrix with the $(i,j)$-th element $\frac{\partial g_i({\theta}_0)}{\partial \theta_j}$;
    \item the Jacobian matrix $\frac{\partial g({\theta}_0)}{\partial {\theta}}$ is of full rank $s$;
    \item the inverse mapping $g^{-1}\colon \Pi_0 \to \Theta$ is continuous at $g({\theta}_0) = {g}_0$;
    \item the mapping $g\colon \Theta \to \Delta_k$ is continuous at every point ${\theta} \in \Theta$;
    \item the function $g\colon \Theta \to \Delta_k$ has continuous second partial derivatives in a~neighbourhood of $\theta_0.$
    \end{enumerate}
It can be shown that the Weibull and Dagum distributions, which are of particular importance in the analysis of income inequality, satisfy the above assumptions (A1), (A3) and (B1)--(B6).
When the parameter space of these distributions is narrowed to a~compact set, condition (A2) holds too.

\subsection{Estimation methods}

In the parametric model considered, we will estimate unknown inequality curves and measures using the plug-in method, i.e., for $p\in(0,1)$, the curve $qZ$ is estimated by a~plug-in estimator $qZ(p;Q_{\hat{\theta}})$ and $qD$ by $qD(p;Q_{\hat{\theta}})$. A~plug-in estimator of $qZ$ has a~value of 1 for $p\in\{0,1\}$, while a~plug-in estimator of $qD$ has a~value of 1 for $p=0$ and a~value of 0 for $p=1$.
In the following, we describe methods to estimate the unknown parameter $\theta$ based on grouped data, which we will use to estimate inequality curves and measures.

Denote
\begin{equation*}
n_i=\sum_{j=1}^{n}\indicator_{[c_{i-1},c_{i})}(X_{j}),
\end{equation*}
where $\indicator_{[c_{i-1},c_{i})}(X_{j})$ equals one, when $X_j\in[c_{i-1},c_{i})$, and zero otherwise.
Let $P=(p_1,\ldots,p_k)$, where 
\begin{equation*}
p_{i}=\frac{n_i}{n},
\end{equation*}
and $G(\theta)=(g_1(\theta),\ldots,g_k(\theta))$, where 
\begin{equation*}
g_i(\theta)=F_\theta(c_i)-F_\theta(c_{i-1}),
\end{equation*}
$i=1,\ldots,k.$

A minimum divergence (MD) estimator $\hat{\theta}$ of the parameter $\theta$ is defined as
\begin{equation}\label{eq:MDestDef}
 \hat{\theta}= \argmin_{\theta\in\Theta} d(P,G(\theta)),
\end{equation}
where $d(\cdot,\cdot)$ is a~divergence between two distributions $P$ and $G(\theta)$.

A~broad class of divergences, namely the class of $\phi$-divergences was proposed in 1963 independently by Csiszár \cite{Csiszar1963} and Morimoto \cite{Morimoto1963}.
However, the idea behind this class was presented in 1961 by Rényi \cite{Renyi1961}.
This class is defined as follows.
\begin{equation*}\label{eq:phidiv}
    d_{\phi}(P,G(\theta)) = \sum_{i=1}^{k} g_i(\theta)\phi\left(\frac{p_i}{g_i(\theta)}\right),
\end{equation*}
for any continuous, convex function $\phi:[0, \infty) \to \mathbb{R}\cup\{\infty\}$, 
where $\phi(1)=0$, $0\phi(0/0)=0$, and $0\phi(t/0) = t \lim_{u\to\infty}(\phi(u)/u)$, for $t>0$
(see, for example, \cite{Menendez1997} or \cite{Pardo2005}).
If
\begin{equation*}
    \phi(x)=\frac{2}{\lambda(\lambda+1)}\left(x^{\lambda+1}-x\right),
\end{equation*}
then $\phi$-divergences, parameterised by the parameter $\lambda \in \mathbb{R}\setminus \{0,-1\}$, create a~class of power divergences
\begin{align*}
    PD^{(\lambda)}(P, G(\theta)) = \frac{2}{\lambda(\lambda+1)} \sum_{i=1}^k p_i \left[ \left(\frac{p_i}{g_i(\theta)}\right)^\lambda - 1 \right],
\end{align*}
where $PD^{(0)}$ and $PD^{(-1)}$ are defined by continuity, and they are of the form
\begin{align}\label{e:KLD}
    PD^{(0)}(P, G(\theta)) &= 2\sum_{i=1}^k p_i \log\left(\frac{p_i}{g_i(\theta)}\right)=:2KLD(P,G(\theta)),
    \end{align}
\begin{align*}
    PD^{(-1)}(P, G(\theta)) &= 2\sum_{i=1}^k g_i(\theta) \log\left(\frac{g_i(\theta)}{p_i}\right).
\end{align*}
The above class of power divergences (also known as the Cressie-Read family of statistics) was proposed by Cressie and Read \cite{Cressie1984} in 1984.

In this paper, we use several examples of divergences to obtain MD estimators of the parameter in the model considered and, consequently, estimators of inequality curves $qZ$, $qD$ and inequality indices $qZI$, $qDI$. 

One of the divergences we consider is the Kullback--Leibler divergence (KLD), i.e., the $PD^{(0)}$ power divergence \eqref{e:KLD}.
It is particularly important in the problem under consideration because in the case of multinomial distribution, minimising the Kullback--Leibler divergence (KLD) is equivalent to maximising the likelihood function; thus, the MD estimator based on KLD is equal to the ML estimator for the grouped data \cite{Read1988}.
In estimating the model parameters, we also propose using the following divergences such that minimising them is equivalent to minimising special cases of power divergence statistics:
\begin{itemize}
    \item
    the squared Hellinger distance (HD) \cite{Matusita1955} 
\begin{equation*}
    HD(P,G(\theta)) = \frac{1}{{2}}{\sum_{i=1}^{k}\left(\sqrt{p_i}-\sqrt{g_i(\theta)}\right)^2},
\end{equation*}
\item
the Pearson chi-squared statistic \cite{Pearson1900}  
\begin{equation*}
    \chi^2_P(P,G(\theta))  = \sum_{i=1}^{k} \frac{(p_i - g_i(\theta))^2}{g_i(\theta)},
\end{equation*}
\item
the Jensen--Shannon divergence (JSD) 
\begin{equation*}
    JSD(P, G(\theta)) = \frac{1}{2} \sum_{i=1}^{k} \left[ p_i \log\left(\frac{2p_i}{p_i + g_i(\theta)}\right) +  g_i(\theta) \log\left(\frac{2g_i(\theta)}{p_i + g_i(\theta)}\right) \right].
\end{equation*}
\end{itemize}
In fact, the Pearson chi-squared statistic is equal to PD with $\lambda=1$.
The functions $\phi$ corresponding to these divergences can be found in Table \ref{tab:phi_function}.
More examples of $\phi$-divergences and the corresponding $\psi$ functions (where $\psi(x) = \phi(x) - (x-1)\phi'(1)$) are given in \cite{Pardo2005} (page 6).

Sometimes minimising two different divergences leads to the same MD estimator. 
For example, the Bhattacharyya distance \cite{Bhattacharyya1946} leads to the same estimate as the Hellinger distance.
For $\lambda=-\tfrac{1}{2}$ the problem of minimising the power divergence statistic $PD^{\left(-\frac{1}{2}\right)}$ is also equivalent to minimising the Hellinger distance. 
For a~comprehensive summary of the topic of $\phi$-divergences and their applications, see, for example, \cite{Basseville2012}.

\begin{table}[h]
    \centering
    \caption{Formulae for the $\phi$ functions corresponding to several $\phi$-divergences studied}
    \label{tab:phi_function}
    \begin{tabular}{c|c}
        $\phi$ function & divergence
        \\
        \hline
        \\[-0.75em]
        ${\frac{1}{2}\left(\sqrt{x}-1\right)^2}$ & Squared Hellinger distance \cite{Matusita1955} \\
        ${ {\frac {1}{2}}\left(x\log x-(x+1)\log \left({\tfrac {x+1}{2}}\right)\right)}$ & Jensen--Shannon divergence \cite{Lin1991} \\
        $x \log x$ & Kullback--Leibler divergence \cite{Kullback1959} \\
        ${\frac{2}{\lambda(\lambda+1)}\left(x^{\lambda+1}-x\right)}$ & 
        Power divergence \cite{Cressie1984} \\
        ${\displaystyle(x-1)^2}$ & Pearson chi-squared statistics \cite{Pearson1900}
    \end{tabular}
\end{table}

\section{Properties of the estimators}\label{sec:properties}

\subsection{Properties of the parameter MD estimators}
Neyman \cite{Neyman1949} defined the Best Asymptotically Normal (BAN) estimator as an asymptotically efficient and asymptotically normal estimator. 
The asymptotic efficiency of an estimator means that its asymptotic variance attains the Cramér--Rao bound.
We recall a~theorem stating that the minimum $\phi$-divergence estimator is a~BAN estimator under Birch regularity conditions (see, for example, Appendix A5 in \cite{Read1988}), and some additional assumptions. 
Note that assumptions (A2) and (B1)--(B5) are the Birch conditions.

\begin{theorem}[Theorem 5.2 in \cite{Pardo2005}]\label{thm:phi_divergence_ban}
    Let $\phi$ be a~twice continuously differentiable function for $x>0$ with $\phi''(1)>0$.
    Under Birch regularity conditions and assuming that the function 
    $g\colon \Theta \to \Delta_k$
 has continuous second partial derivatives in a~neighbourhood of $\theta_0$, the minimum $\phi$-divergence estimator is a~BAN estimator, namely
 \begin{equation}
     \sqrt{n} \left( \hat{\theta}_n^\phi - \theta_0 \right) \xrightarrow[]{d}
     \mathcal{N}\left({\bf 0}, I^{-1}_{\theta_0} \right),
 \end{equation}
where ${\bf 0}$ is the vector of zeroes and $I_{\theta}$ is the Fisher information matrix of the multinomial model with $k$ cells.
\end{theorem}

\begin{theorem}
Under assumptions (A1)--(A2) and (B1)--(B6), the MD parameter estimators based on PD (in special cases KLD, HD, $\chi_{P}^{2}$) and JSD are BAN.
\end{theorem}
   \begin{proof}
It is enough to show that for each divergence considered, the function $\phi$ satisfies the assumption of Theorem \ref{thm:phi_divergence_ban}.
Table \ref{tab:phi_function} shows the formulae for functions $\phi$ corresponding to several $\phi$-divergences, including the family of PD for every $\lambda\in\mathbb{R}\setminus \{0,-1\}$,  a~limit case of PD for $\lambda=0$ (KLD), and JSD.
All these functions can be differentiated twice continuously with the second derivative strictly larger than 0, for every $x>0$ (thus also for $x=1$). 
That is, in the case of PD, for every $\lambda\in\mathbb{R}\setminus \{0,-1\}$ we have 
\begin{align*}
    \phi''(x) = 2x^{\lambda-1},
\end{align*}
for KLD we have
\begin{align*}
    \phi''(x) = \frac{1}{x},
\end{align*}
for HD we have
\begin{align*}
    \phi''(x) = \frac{1}{4}x^{-\frac{3}{2}},
\end{align*}
for $\chi_{P}^{2}$ we have
\begin{align*}
\phi''(x) = 2,
\end{align*}
and for JSD we have
\begin{align*}
    \phi''(x) = \frac{1}{2x(x+1)}.
\end{align*}
Therefore, the functions $\phi$ corresponding to PD, KLD, HD, $\chi_{P}^{2}$, and JSD satisfy the assumptions of Theorem \ref{thm:phi_divergence_ban}, hence, in the model considered, the MD estimators based on these $\phi$-divergences are BAN.
    \end{proof}
\begin{remark}\label{rem:BAN}
  In the model considered, the MD method, based on KLD, HD, $\chi_{P}^{2}$, JSD, leads to asymptotically equivalent estimators.  
\end{remark}

\subsection{Properties of the inequality curves and measures estimator}

From this point, by $\hat{\theta}_{n}^{\phi}$ we denote an estimator of $\theta$ obtained by minimising a~specific $\phi$-divergence, based on a~sample of size $n$, for example, $\hat{\theta}_{n}^{HD}$ minimises the Hellinger distance, and so on.
A~plug-in estimator $\widehat{Q}^{\phi}_n:=Q_{\hat{\theta}_{n}^{\phi}}$ of the quantile function $Q_{\theta}$ is obtained by inserting ${\hat{\theta}_{n}^{\phi}}$ into the formula for $Q_{\theta}$.
Analogously plug-in estimators $\widehat{qZ}^{\phi}_{n}(\cdot):=qZ(\cdot;\widehat{Q}_{n}^{\phi})$ of $qZ(\cdot;Q_{\theta})$ and $\widehat{qZI}^{\phi}_{n}:=qZI(\widehat{Q}_{n}^{\phi})$ of $qZI(Q_{\theta})$ can be defined, 
and similarly $\widehat{qD}^{\phi}_{n}(\cdot):=qD(\cdot;\widehat{Q}_{n}^{\phi})$ is an estimator of $qD(\cdot;Q_{\theta})$ and $\widehat{qDI}^{\phi}_{n}:=qDI(\widehat{Q}_{n}^{\phi})$ is an estimator of $qDI(Q_{\theta})$.

\begin{theorem}\label{thm:pointwise_convergence_qz}
    Let $\hat{\theta}_{n}^{\phi}$ be a~BAN estimator of the parameter $\theta_0$ 
    in the model described in Section \ref{sec:model}. 
    Then, for every $p\in[0,1]$, we have
    \begin{equation*}
        \widehat{qZ}^{\phi}_{n}(p) \xrightarrow{P} qZ(p;Q_{\theta_0})
    \end{equation*}
    and
        \begin{equation*}
        \widehat{qD}^{\phi}_{n}(p) \xrightarrow{P} qD(p;Q_{\theta_0})
    \end{equation*}
    when $n$ tends to infinity.
\end{theorem}

\begin{proof}
By assumption, $\hat{\theta}_{n}^{\phi}$ is BAN, so in particular it is a~consistent estimator of the true parameter $\theta_0$.
In the model under consideration, for each $p\in[0,1]$, the functions $qZ(p;Q_{\theta})$, $qD(p;Q_{\theta})$ are continuous functions of the true argument $\theta$ (see assumption (A3)), and the theorem follows from the Continuous Mapping Theorem (see Theorem 2.3 in \cite{VanDerVaart1998}).
\end{proof}

\begin{theorem}\label{thm:qd_conv_eta}
    Let $\hat{\theta}_{n}^{\phi}$ be a~BAN estimator of the parameter $\theta$ in the model described in Section \ref{sec:model}. Then
            \begin{equation*}
        \sup_{p\in[0,1]}
        \left|\widehat{qD}^{\phi}_{n}(p) - qD(p;Q_{\theta_0})\right|
        \xrightarrow{P} 0,
    \end{equation*}
    when $n$ tends to infinity.
    \end{theorem}

\begin{proof}
    Since $Q_{\theta_0}$ is continuous, the function $qD(\cdot;Q_{\theta_0})$ is also continuous.
    Thus, the uniform convergence in probability of the sequence $\left(\widehat{qD}_{n}^{\phi}\right)$ to $qD$ follows from the fact that a~sequence of monotonic functions that converges in probability at every point on a~compact set to a~continuous function converges to it uniformly in probability (see Lemma 19 in \cite{Koriyama2026}).
\end{proof}    

For the proof of the uniform convergence in probability of the plug-in estimators of $qZ$, we need an additional assumption 
and the following lemma 
because the function $qZ$ is not monotone.

\begin{lemma}\label{lem:cont_fun}
    If $f$ is a~continuous, real function of $(x;\theta)$ on a~compact set $K\times\Theta$, where $K\subset\R$ and $\Theta\subset\R^s$ and $\hat{\theta}_{n}$ converges in probability to $\theta_0$, then for $f_n(x):=f(x;\hat{\theta}_n)$ and $f_0(x)=f(x;{\theta_0})$ we have
    \begin{equation*}
        \sup_{x\in K}
        \left|f_n(x) - f_0(x)\right|
        \xrightarrow{P} 0,
    \end{equation*}
       when $n$ tends to infinity.
\end{lemma}

\begin{proof}
    We define a~nondecreasing deterministic function of $\delta$
    \begin{equation*}
      m(\delta) \;:=\; \sup\Bigl\{\,
  \bigl|f(x;{\theta_1})-f(x;{\theta_2})\bigr|\colon\;
  x\in K,\ \theta_1,\theta_2\in\Theta,\ ||\theta_1-\theta_2||\le\delta
  \,\Bigr\}.
    \end{equation*}
    The function $f$ is continuous on a~compact set $K\times\Theta$, thus it is uniformly continuous. Hence, for all $\varepsilon>0$ there exists $\delta_{\varepsilon}>0$ such that
    \begin{equation*}
        |x_1-x_2| + ||\theta_1-\theta_2|| \leq \delta_{\varepsilon} \Rightarrow 
        |f(x_1;{\theta_1}) - f(x_2;Q_{\theta_2})| \leq {\varepsilon}.
    \end{equation*}
    Thus, taking $x:=x_1=x_2$, we get
    \begin{equation*}
        ||\theta_1-\theta_2|| \leq \delta_{\varepsilon} \Rightarrow 
        |f(x;{\theta_1}) - f(x;{\theta_2})| \leq {\varepsilon},
    \end{equation*}
    for every $x\in K$. Hence $m(\delta) \leq {\varepsilon}$ for every $\delta \leq \delta_{\varepsilon}$. As $\varepsilon$ was arbitrary, we have $m(\delta)\to0$ when $\delta\to0$.

    Now fix $n$ and an elementary event $\omega$ in the probability space $\Omega$.
    We can see that for every $\omega$ we have
    \begin{equation}\label{eq:def_sn}
        S_n(\omega) := \sup_{x\in K} 
        \left|
        f(x;{\hat{\theta}_{n}})-
        f(x;{\theta_0})
        \right| \leq m\left(\left|\left|{\hat{\theta}_{n}}-\theta_0\right|\right|\right).
    \end{equation}
    Let $\eta>0$. We know that $m(\delta)\xrightarrow[]{\delta\to0} 0$, so we can fix $\delta_{\eta} > 0$ such that $m(\delta_{\eta})<\eta$.
For $\omega\in\left\{\left|\left|{\hat{\theta}_{n}}-\theta_0\right|\right|\leq\delta_{\eta}\right\}$ the monotonicity of $m$ and \eqref{eq:def_sn} give
    \begin{equation*}
        S_n(\omega) \leq 
        m\left(\left|\left|{\hat{\theta}_{n}}-\theta_0\right|\right|\right) \leq m(\delta_{\eta})  < \eta,
    \end{equation*}
    thus $\left\{\left|\left|{\hat{\theta}_{n}}-\theta_0\right|\right|\leq\delta_{\eta}\right\} \subseteq \{S_n\leq\eta\}$,
    hence $ \{S_n>\eta\}  \subseteq  \left\{\left|\left|{\hat{\theta}_{n}}-\theta_0\right|\right|>\delta_{\eta}\right\}$,
    therefore
    \begin{equation}\label{eq:proof_3_7_end}
        P(S_n>\eta) \leq P\left(\left|\left|{\hat{\theta}_{n}}-\theta_0\right|\right|>\delta_{\eta}\right).
    \end{equation}
 The RHS of \eqref{eq:proof_3_7_end} converges to 0 since $\hat{\theta}_{n}$ converges in probability to $\theta_0$, thus $S_n \xrightarrow[n\to\infty]{P}0$, which ends the proof.
\end{proof}

\begin{theorem}\label{thm:qz_conv_eta}
    Let $\hat{\theta}_{n}^{\phi}$ be a~BAN estimator of the parameter $\theta$ in the model described in Section \ref{sec:model}.
    If $qZ$ is a~continuous function of $(p,\theta)$, then
    \begin{equation*}
        \sup_{p\in[0,1]}
        \left|\widehat{qZ}^{\phi}_{n}(p) - qZ(p;Q_{\theta_0})\right|
        \xrightarrow{P} 0,
    \end{equation*}
       when $n$ tends to infinity.
    \end{theorem}

\begin{proof}
    The estimator $\hat{\theta}_{n}^{\phi}$ is BAN, hence it converges in probability to $\theta_0$.
    Since we assumed that $qZ$ is a~continuous function of $(p,\theta)$, by Lemma \ref{lem:cont_fun} we get the thesis.
\end{proof}

The following theorem establishes asymptotic normality and asymptotic efficiency of estimators defined in Section~\ref{sec:methodology}

\begin{theorem}\label{thm:qzi_qdi_as_norm}
 Let $\hat{\theta}_{n}^{\phi}$ be a~BAN estimator of the parameter $\theta$ in the model described in Section~\ref{sec:model} and 
$v_1(\theta):=qZI(Q_{\theta})$, $v_2(\theta):=qDI(Q_{\theta})$.
 If $v_{i}$ is differentiable at $\theta_{0}$ and $\nabla v_i (\theta_0) \neq {\bf 0}$, where ${\bf 0}$ is a~vector of zeroes,
 for $i=1,2,$ 
 then, the estimators $\widehat{qZI}^{\phi}_{n}=v_{1}\left({\hat{\theta}_{n}^{\phi}}\right)$ and 
$\widehat{qDI}^{\phi}_{n}=v_{2}\left({\hat{\theta}_{n}^{\phi}}\right)$ are asymptotically normal estimators of $qZI\left(Q_{\theta_0}\right)$ and $qDI\left(Q_{\theta_0}\right)$, respectively, namely
  \begin{equation}
     \sqrt{n} \left(v_{i}\left(\hat{\theta}_n^\phi\right) - v_{i}(\theta_0) \right) \xrightarrow[]{d}
     \mathcal{N}\left(0, 
     \nabla v_i (\theta_0) 
     I^{-1}_{\theta_0}
     (\nabla v_i (\theta_0) )^{\top}
     \right),
 \end{equation}
 for $i=1,2$.
\end{theorem}
\begin{proof}
The theorem follows from the asymptotic normality of $\hat{\theta}_{n}^{\phi}$ and the delta method (see Theorem 3.1 in \cite{VanDerVaart1998}).
\end{proof}

\begin{remark}
    Theorem \ref{thm:qzi_qdi_as_norm} allows us to construct asymptotically pointwise confidence intervals for $qZI$ and $qDI$.
\end{remark}

\section{Simulation study}\label{sec:simulation_study}

The main goal of this study is to compare the estimators of quantile inequality curves and measures obtained using the MD method with various divergences.
Since we focus here on using these inequality curves and measures on grouped income data, in our simulation study we choose probability distributions that are often used to fit income data.
Thus, we consider distributions which are special cases of the generalised beta distribution of the second kind (GB2), introduced by McDonald \cite{McDonald1984}. The GB2 distribution can capture a~wide range of distributional shapes, including heavy tails and skewness. 
For this reason, it has become a~popular choice for modelling income and has many other applications in actuarial science and finance (see, for example, \cite{McDonald1995} or \cite{Kleiber2003}).
The probability density function (pdf) of GB2 is given by
\begin{equation*}\label{e:fGB2}
f(x; a, b, r, q) = \frac{a x^{a r - 1}}{b^{a r} B(r, q) \left[1 + \left(\frac{x}{b}\right)^a\right]^{r + q}}, \quad x > 0,
\end{equation*}
where  $B(r, q)$ is the Beta function, $a, r, q > 0$ are the shape parameters and $b>0$ is the scale parameter.
In the problem of income analysis, one of the most popular distributions is the Dagum distribution
$\mathcal{D}(a, r, b)$ with the pdf
\begin{equation*}\label{e:fDagum}
{\displaystyle f(x;a,r,b)={\frac {ar\left({\tfrac {x}{b}}\right)^{ar}}{x\left[1+\left({\tfrac {x}{b}}\right)^{a}\right]^{r+1}}},}
\end{equation*}
which is a~special case of GB2 with $q=1$.

The special and limiting cases of GB2 with two parameters include Pareto, Weibull, gamma, and lognormal distributions, among others (see Figure 2 in \cite{McDonald1995}). 
In the simulations, in addition to the Dagum distribution, we also consider the Weibull distribution, which is given by the probability density function
\begin{equation*}
{\displaystyle f(x;\sigma,\beta )=\frac{\beta}{\sigma}\left({\frac{x}{\sigma}}\right)^{\beta-1}\exp\left[{-\left(\frac{x}{\sigma}\right)^{\beta}}\right]},
\end{equation*}
where $x>0$, $\beta>0$, and $\sigma>0$.

The values of the estimators $\hat{\theta}^\phi_n$ of the parameter $\theta$ were computed using an implementation of the \textit{L-BFGS-B} algorithm \cite{Byrd1995} in the \textit{minimize} function of the \textit{scipy.optimize} module in Python.

In this section, the abbreviations HD, JSD, etc. refer to the estimators of the parameters of a~certain distribution obtained by minimising the corresponding divergence statistic between the empirical probabilities and the parametric family of distribution functions.
In tables and figures addressing the problem of estimating the quantile inequality curves and measures, these abbreviations stand for the corresponding plug-in estimators of $qZ$, $qD$, $qZI$ and $qDI$.
We consider the power divergence with $\lambda=\frac{2}{3}$, i.e. $PD^{\left(\frac 23\right)}$, since this value was recommended by Cressie and Read (see Chapter 6 in \cite{Read1988}) due to its good performance in goodness-of-fit tests or estimation of the vector of probabilities, compared to other choices of $\lambda$.

\begin{table}[h]
\caption{Edges of the groups used for grouping the observations drawn from the Weibull distribution in the simulation study}
\label{tab:bin_edges_weibull}
\begin{tabular}{l|rrrrrr}
\toprule
$c_j \backslash \beta$ & 0.5 & 0.8 & 1.0 & 1.2 & 1.6 & 2.0 \\
\midrule
$c_1$ & 0.1 & 0.2 & 1.0 & 1.0 & 1.0 & 1.5 \\
$c_2$ & 0.5 & 0.5 & 1.5 & 1.5 & 1.5 & 2.5 \\
$c_3$ & 1.2 & 1.2 & 2.5 & 2.5 & 2.5 & 3.5 \\
$c_4$ & 2.5 & 2.5 & 3.5 & 3.5 & 3.5 & 4.0 \\
$c_5$ & 5.0 & 3.5 & 5.0 & 5.0 & 4.0 & 5.0 \\
$c_6$ & 7.5 & 5.0 & 7.5 & 6.5 & 5.0 & 5.5 \\
$c_7$ & 10.0 & 7.5 & 10.0 & 9.0 & 6.0 & 6.0 \\
$c_8$ & 20.0 & 12.5 & 15.0 & 11.0 & 8.0 & 7.0 \\
$c_9$ & 50.0 & 20.0 & 20.0 & 15.0 & 10.0 & 8.5 \\
\bottomrule
\end{tabular}
\end{table}

\begin{table}[h]
\caption{True probabilities that an observation drawn from the Weibull distribution with a~shape parameter $\beta$ and scale parameter $\sigma=5$ falls into $i$-th group in the simulation study}
\centering
\begin{tabular}{c|cccccc}
\toprule
$i \backslash \beta$ & 0.5 & 0.8 & 1.0 & 1.2 & 1.6 & 2.0 \\
\midrule
1 & 0.132 & 0.073 & 0.181 & 0.135 & 0.073 & 0.086 \\
2 & 0.139 & 0.073 & 0.078 & 0.075 & 0.062 & 0.135 \\
3 & 0.122 & 0.134 & 0.134 & 0.143 & 0.145 & 0.166 \\
4 & 0.113 & 0.156 & 0.110 & 0.126 & 0.151 & 0.085 \\
5 & 0.125 & 0.092 & 0.129 & 0.153 & 0.072 & 0.159 \\
6 & 0.074 & 0.104 & 0.145 & 0.114 & 0.129 & 0.070 \\
7 & 0.051 & 0.117 & 0.088 & 0.122 & 0.106 & 0.061 \\
8 & 0.108 & 0.126 & 0.086 & 0.056 & 0.142 & 0.096 \\
9 & 0.093 & 0.077 & 0.031 & 0.052 & 0.072 & 0.085 \\
10& 0.042 & 0.048 & 0.018 & 0.024 & 0.048 & 0.056 \\
\bottomrule
\end{tabular}
\label{tab:prob_weibull_groups}
\end{table}

\subsection{Estimation of $qZI$ and $qDI$ for samples from the Weibull distribution}\label{sec:qzi_est_weibull}

The experiment is conducted as follows.
We draw a~sample $\mathbf{x}$ of size $n=100$ from the Weibull distribution. 
The sample is then grouped into $k=10$ groups with edges defined as in Table \ref{tab:bin_edges_weibull}, which shows the vectors of the edges $(c_1, c_2, \ldots, c_9)$ of the groups used in the simulation study when the observations were sampled from the Weibull distribution for each value of the shape parameter $\beta$.
Then we try to fit the Weibull distribution to the observed frequencies of the groups using the statistics described in the previous sections.
We repeat this for $M=1000$ samples. 
For each sample, $\hat{\theta}^\phi_n$ is computed for various $\phi$, namely HD, JSD, KLD, PD with $\lambda=\frac{2}{3}$ and $\chi^2_P$, as defined in Section \ref{sec:methodology}.
Moreover, $qZI$ and $qDI$ are estimated using the plug-in estimator obtained by inserting $\hat{\theta}^\phi_n$ into the formula for $qZI$ and $qDI$, respectively.
Similarly, the plug-in estimator of $qDI$ is constructed.
In Table \ref{tab:prob_weibull_groups} we can see the true probabilities of falling into each bin for six different values of the shape parameter $\beta$ that vary between $0.5$ and $2$.
Throughout this simulation study, all samples drawn from the Weibull distribution are drawn from the Weibull distribution with scale parameter $\sigma=5$

Table \ref{tab:beta_mse_weibull_weibull} shows the MSEs of the estimators of the shape parameter $\beta$.
In Tables \ref{tab:qzi_mse_weibull_weibull} and \ref{tab:qdi_mse_weibull_weibull} we can see the MSEs of the plug-in estimators of $qZI$ and $qDI$, respectively, for samples drawn from the Weibull distribution, assuming the Weibull distribution.
The $\chi^2_P$ divergence gave the lowest MSE of the shape parameter estimators and both inequality measures for most of the shape parameter values, followed by the power divergence with $\lambda=\frac{2}{3}$.

\begin{table}[h]
\caption{MSE (multiplied by 1000) of the estimators of the shape parameter of the Weibull distribution with shape parameter $\beta$ based on samples of size $n=100$ drawn from the Weibull distribution divided into 10 groups. PD was used with $\lambda=\frac{2}{3}$.}
\centering
\begin{tabular}{l|rrrrr}
\toprule
$\beta$ & HD & JSD & KLD & PD & $\chi^2_P$ \\
\midrule
0.5 & 2.743 & 2.717 & 2.640 & \textbf{2.600} & 2.603 \\
0.8 & 5.666 & 5.605 & 5.244 & 4.950 & \textbf{4.866} \\
1.0 & 10.262 & 9.702 & 8.688 & 8.035 & \textbf{7.882} \\
1.2 & 12.952 & 12.539 & 11.479 & 10.712 & \textbf{10.513} \\
1.6 & 22.855 & 22.513 & 21.007 & 19.767 & \textbf{19.415} \\
2.0 & 35.244 & 35.596 & 33.728 & 32.976 & \textbf{32.914} \\
\bottomrule
\end{tabular}
\label{tab:beta_mse_weibull_weibull}
\end{table}

\begin{table}[h]
\caption{MSE (multiplied by 1000) of estimates of $qZI$ based on estimators of shape parameter of the Weibull distribution with shape parameter $\beta$ based on samples of size $n=100$ drawn from the Weibull distribution divided into 10 groups. PD was used with $\lambda=\frac{2}{3}$.}
\centering
\begin{tabular}{l|l|rrrrr}
\toprule
$qZI$ & $\beta$ & HD & JSD & KLD & PD & $\chi^2_P$ \\
\midrule
0.968 & 0.5 & 0.109 & 0.108 & 0.106 & \textbf{0.105} & 0.106 \\
0.890 & 0.8 & 0.465 & 0.460 & 0.430 & 0.406 & \textbf{0.399} \\
0.833 & 1.0 & 0.802 & 0.760 & 0.684 & 0.635 & \textbf{0.623} \\
0.777 & 1.2 & 0.876 & 0.850 & 0.782 & 0.731 & \textbf{0.717} \\
0.681 & 1.6 & 1.042 & 1.029 & 0.965 & 0.911 & \textbf{0.895} \\
0.602 & 2.0 & 1.066 & 1.071 & 1.010 & 0.974 & \textbf{0.965} \\
\bottomrule
\end{tabular}
\label{tab:qzi_mse_weibull_weibull}
\end{table}

\begin{table}[h]
\caption{MSE (multiplied by 1000) of estimates of $qDI$ based on estimators of shape parameter of the Weibull distribution with shape parameter $\beta$ based on samples of size $n=100$ drawn from the Weibull distribution divided into 10 groups. PD was used with $\lambda=\frac{2}{3}$.}
\centering
\begin{tabular}{l|l|rrrrr}
\toprule
$qDI$  & $\beta$ & HD & JSD & KLD & PD & $\chi^2_P$ \\
\midrule
0.835 & 0.5 & 0.253 & 0.251 & 0.244 & 0.240 & \textbf{0.239} \\
0.751 & 0.8 & 0.371 & 0.368 & 0.345 & 0.326 & \textbf{0.320} \\
0.702 & 1.0 & 0.526 & 0.500 & 0.451 & 0.419 & \textbf{0.412} \\
0.658 & 1.2 & 0.533 & 0.517 & 0.476 & 0.445 & \textbf{0.437} \\
0.583 & 1.6 & 0.611 & 0.603 & 0.565 & 0.534 & \textbf{0.524} \\
0.523 & 2.0 & 0.635 & 0.638 & 0.602 & 0.581 & \textbf{0.576} \\
\bottomrule
\end{tabular}
\label{tab:qdi_mse_weibull_weibull}
\end{table}

\begin{figure}
    \centering
    \includegraphics[width=1.0\linewidth]{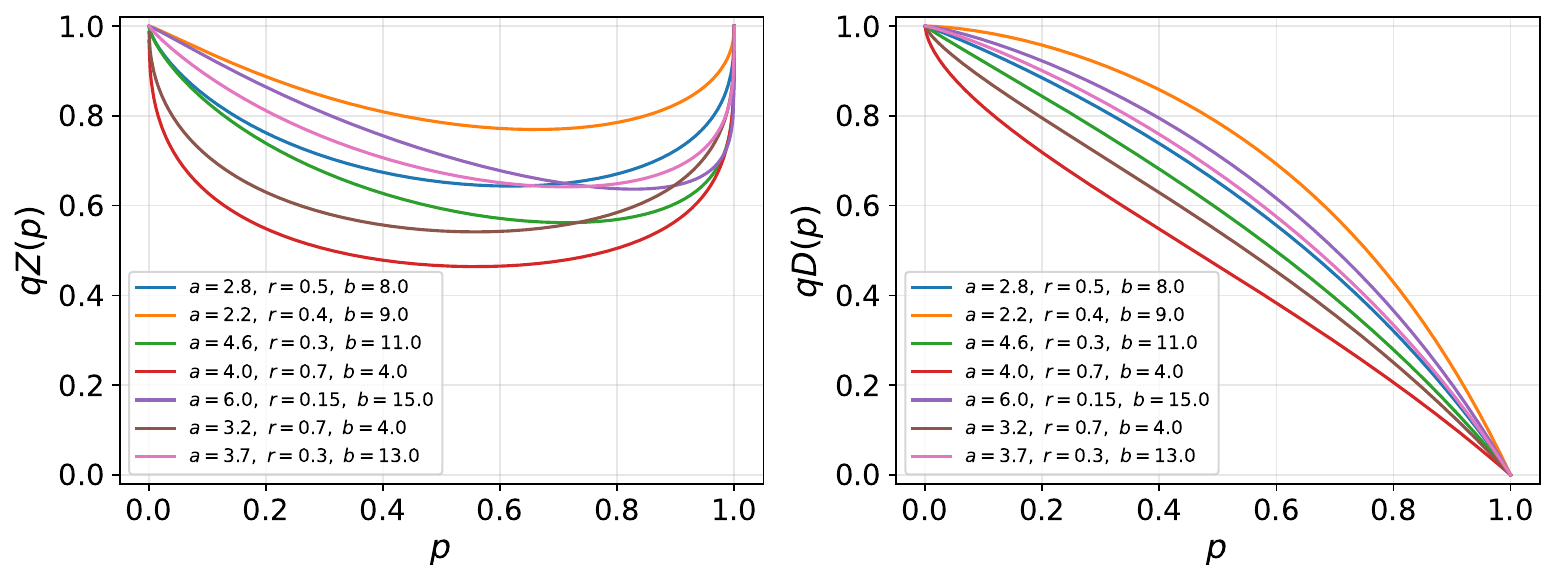}
    \caption{Example of $qZ$ and $qD$ curves for random variables from the Dagum distribution with various sets of parameters considered in the simulation study. The shape of both curves depends only on shape parameters.}
    \label{fig:qZ_qD_Dagum_example}
\end{figure}

\subsection{Estimation of $qZI$ and $qDI$ for samples from the Dagum distribution} \label{sec:qzi_est_dagum}

\begin{table}
\caption{Edges of the groups used for grouping the observations drawn from the Dagum distribution in the simulation study}
\label{tab:bin_edges_dagum}
\begin{tabular}{l|rrrrrrr}
\toprule
$a$ & 2.8 & 2.2 & 4.6 & 4.0 & 6.0 & 3.2 & 3.7 \\
$r$ & 0.5 & 0.4 & 0.3 & 0.7 & 0.15 & 0.7 & 0.3 \\
$b$  & 8.0 & 9.0 & 11.0 & 4.0 & 15.0 & 4.0 & 13.0 \\
\midrule
$c_1$ & 1.0 & 1.0 & 1.0 & 1.5 & 1.0 & 1.5 & 1.0 \\
$c_2$ & 1.5 & 1.5 & 1.5 & 2.0 & 1.5 & 2.0 & 1.5 \\
$c_3$ & 2.5 & 2.5 & 2.5 & 2.5 & 2.5 & 2.5 & 2.5 \\
$c_4$ & 3.5 & 3.5 & 3.5 & 3.0 & 3.5 & 3.0 & 3.5 \\
$c_5$ & 5.0 & 5.0 & 5.0 & 3.5 & 5.0 & 3.5 & 5.0 \\
$c_6$ & 7.5 & 7.5 & 7.5 & 5.0 & 7.5 & 5.0 & 7.5 \\
$c_7$ & 10.0 & 10.0 & 10.0 & 6.0 & 10.0 & 6.0 & 10.0 \\
$c_8$ & 15.0 & 15.0 & 15.0 & 7.0 & 15.0 & 7.0 & 15.0 \\
$c_9$ & 20.0 & 20.0 & 20.0 & 8.0 & 20.0 & 8.0 & 20.0 \\
\bottomrule
\end{tabular}
\end{table}

\begin{table}
\caption{True probabilities that an observation drawn from the Dagum distribution with parameters $(a,r,b)$ falls into $i$-th group in the simulation study}
\label{tab:bin_probs_dagum}
\begin{tabular}{l|rrrrrrr}
\toprule
$a$ & 2.8 & 2.2 & 4.6 & 4.0 & 6.0 & 3.2 & 3.7 \\
$r$ & 0.5 & 0.4 & 0.3 & 0.7 & 0.15 & 0.7 & 0.3 \\
$b$  & 8.0 & 9.0 & 11.0 & 4.0 & 15.0 & 4.0 & 13.0 \\
\midrule
$1$ & 0.054 & 0.144 & 0.037 & 0.063 & 0.087 & 0.108 & 0.058 \\
$2$ & 0.041 & 0.061 & 0.027 & 0.074 & 0.038 & 0.089 & 0.033 \\
$3$ & 0.097 & 0.111 & 0.065 & 0.105 & 0.073 & 0.106 & 0.069 \\
$4$ & 0.107 & 0.099 & 0.076 & 0.126 & 0.071 & 0.112 & 0.072 \\
$5$ & 0.160 & 0.126 & 0.129 & 0.130 & 0.102 & 0.107 & 0.111 \\
$6$ & 0.215 & 0.153 & 0.228 & 0.288 & 0.163 & 0.235 & 0.180 \\
$7$ & 0.133 & 0.098 & 0.193 & 0.095 & 0.151 & 0.088 & 0.155 \\
$8$ & 0.117 & 0.102 & 0.182 & 0.050 & 0.216 & 0.053 & 0.192 \\
$9$ & 0.040 & 0.045 & 0.044 & 0.027 & 0.074 & 0.033 & 0.076 \\
${10}$ & 0.036 & 0.062 & 0.018 & 0.042 & 0.024 & 0.070 & 0.054 \\
\bottomrule
\end{tabular}
\end{table}

Similarly as in Section \ref{sec:qzi_est_weibull}, we compare the estimators of $qZI$ and $qDI$ based on various statistics based on samples of size $n=100$ drawn from the Dagum distribution.
These samples are grouped into 10 groups with edges defined as in Table \ref{tab:bin_edges_dagum}, which shows the edge vectors $(c_1, c_2, \ldots, c_9)$ that define the groups used in the simulation study for various combinations of parameters $(a,r,b)$ of the Dagum distribution.
The choice of the sets of parameters was inspired by the estimates of the parameters of the Dagum distribution fitted to income data (gathered in various countries) by Bandourian et al. \cite{Bandourian2003}.
Table \ref{tab:bin_probs_dagum} shows the true probabilities of falling into each bin for all the sets of parameters considered in the simulation study with samples from the Dagum distribution.

Figure \ref{fig:qZ_qD_Dagum_example} shows the true curves $qZ$ and $qD$ for the Dagum distribution with the chosen sets of parameters.
Tables \ref{tab:qzi_mse_dagum_dagum} and \ref{tab:qdi_mse_dagum_dagum} show the MSE of the plug-in estimators of $qZI$ and $qDI$, respectively, for samples drawn with different sets of parameters of the Dagum distribution, when the Dagum distribution is assumed.
As in the Weibull distribution case, $\chi^2_P$ followed by PD with $\lambda=\frac{2}{3}$ led to estimators of $qZI$ and $qDI$ with lowest MSE in almost all sets of parameters considered from the Dagum distribution.
The analysis of boxplots (see Figure \ref{fig:dagum_dagum_46_qzi}) confirms that PD with $\lambda=\frac{2}{3}$ and $\chi^2_P$ are a~slightly more safe choice in terms of outlying observations and lead to estimates less distant from the true value of the indices $qZI$ than the other estimators.
Boxplots with further examples are given in the GitHub repository mentioned in Section \ref{sec:data_and_code}.

\begin{figure}
    \centering
    \includegraphics[width=1.0\linewidth]{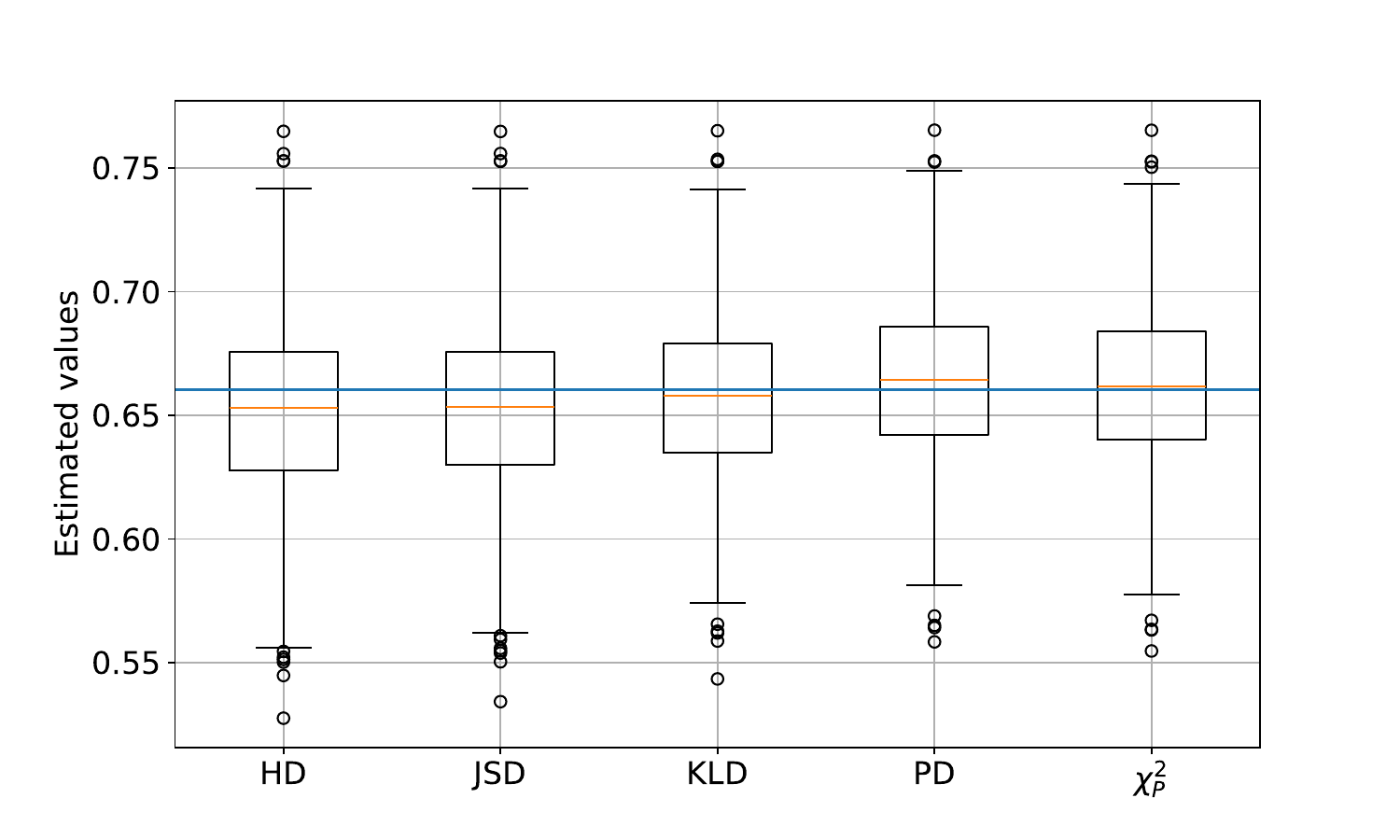}
    \caption{A comparison of the estimated values of $qZI$ for 1000 samples of size 100 drawn from the Dagum distribution $\mathcal{D}(4.6, 0.3, 11.0)$ and divided into 10 groups. Estimates obtained with different estimators assuming the Dagum distribution. The blue horizontal line represents the true value of the estimated index $qZI$.}
    \label{fig:dagum_dagum_46_qzi}
\end{figure}

\begin{table}
\caption{MSE (multiplied by 1000) of estimates of $qZI$ based on estimators of the parameters of the Dagum distribution based on samples of size $n=100$ drawn from the Dagum distribution with parameters $\theta=(a,r,b)$ divided into 10 groups. PD was used with $\lambda=\frac{2}{3}$.}
\label{tab:qzi_mse_dagum_dagum}
\begin{tabular}{l|c|rrrrr}
\toprule
$qZI$ & $(a,r,b)$ & HD & JSD & KLD & PD & $\chi^2_P$ \\
\midrule
0.722 & (2.8, 0.5, 8.0) & 1.100 & 1.076 & 0.984 & 0.914 & \textbf{0.895} \\
0.837 & (2.2, 0.4, 9.0) & 0.823 & 0.818 & 0.791 & 0.768 & \textbf{0.762} \\
0.660 & (4.6, 0.3, 11.0) & 1.423 & 1.341 & 1.163 & 1.070 & \textbf{1.066} \\
0.539 & (4.0, 0.7, 4.0) & 1.094 & 1.073 & 0.991 & 0.922 & \textbf{0.899} \\
0.753 & (6.0, 0.15, 15.0) & 1.286 & 1.276 & 1.172 & 1.095 & \textbf{1.070} \\
0.618 & (3.2, 0.7, 4.0) & 1.230 & 1.218 & 1.172 & 1.127 & \textbf{1.110} \\
0.734 & (3.7, 0.3, 13.0) & 1.187 & 1.168 & 1.055 & 0.988 & \textbf{0.969} \\
\bottomrule
\end{tabular}
\end{table}

\begin{table}
\caption{MSE (multiplied by 1000) of estimates of $qDI$ based on estimators of the parameters of the Dagum distribution based on samples of size $n=100$ drawn from the Dagum distribution with parameters $\theta=(a,r,b)$ divided into 10 groups. PD was used with $\lambda=\frac{2}{3}$.}
\label{tab:qdi_mse_dagum_dagum}
\begin{tabular}{l|c|rrrrr}
\toprule
$qDI$ & $(a,r,b)$ & HD & JSD & KLD & PD & $\chi^2_P$ \\
\midrule
0.605 & (2.8, 0.5, 8.0) & 0.645 & 0.633 & 0.585 & 0.548 & \textbf{0.538} \\
0.697 & (2.2, 0.4, 9.0) & 0.624 & 0.621 & 0.596 & 0.574 & \textbf{0.567} \\
0.563 & (4.6, 0.3, 11.0) & 0.896 & 0.848 & 0.744 & 0.689 & \textbf{0.687} \\
0.465 & (4.0, 0.7, 4.0) & 0.645 & 0.633 & 0.586 & 0.547 & \textbf{0.534} \\
0.647 & (6.0, 0.15, 15.0) & 0.941 & 0.940 & 0.866 & 0.815 & \textbf{0.802} \\
0.525 & (3.2, 0.7, 4.0) & 0.711 & 0.705 & 0.678 & 0.651 & \textbf{0.641} \\
0.619 & (3.7, 0.3, 13.0) & 0.776 & 0.765 & 0.696 & 0.658 & \textbf{0.648} \\
\bottomrule
\end{tabular}
\end{table}

\subsection{Estimation of quantile inequality curves}

We computed the mean integrated squared error (MISE) of the minimum $\phi$-divergence plug-in estimators of $qZ$ and $qD$ for samples from the Weibull and Dagum distribution.
For a~plug-in estimator $\widehat{qZ}_n^\phi$ of $qZ$ we define MISE as follows,
\begin{align*}
    \text{MISE}\left(\widehat{qZ}_{n}^\phi\right)=\mathbb{E}
    \left(
    \int_0^1 
    \left[qZ(p;\widehat{Q}_n^\phi)-qZ(p;Q_{\theta_0})\right]^2\,dp
    \right).
\end{align*}
For an estimator $\widehat{qD}_n^\phi$ of $qD$, MISE is defined analogously.
The estimated MISE of the plug-in estimators of $qZ$ and $qD$, respectively, based on minimum $\phi$-divergence estimators of $\theta$ are presented
in Tables \ref{tab:mise_weibull_weibull_qz} and \ref{tab:mise_weibull_weibull_qd} for samples drawn from the Weibull distribution.
Analogously for the Dagum distribution, MISE of the plug-in estimators of $qZ$ and $qD$, respectively, are presented
in Tables \ref{tab:mise_dagum_dagum_qz} and \ref{tab:mise_dagum_dagum_qd}.
The integral required to obtain MISE was computed numerically using the \textit{quad} function from \textit{scipy.integrate} in Python.
Assuming the Dagum distribution, $\chi^2_P$ gives the lowest estimated MISE, followed by PD with $\lambda=\frac{2}{3}$ in all cases, except for MISE of the estimators of $qZ$, where PD with $\lambda=\frac{2}{3}$ leads to an estimator with the lowest MISE for one combination of parameters (when $r$ is low).

\begin{table}
\caption{MISE (multiplied by 1000) of the minimum $\phi$-divergence plug-in estimators of $qZ$ for 1000 samples drawn from the Weibull distribution, assuming the Weibull distribution}
\label{tab:mise_weibull_weibull_qz}
\begin{tabular}{l|rrrrr}
\toprule
$\beta$ & HD & JSD & KLD & PD & $\chi^2_P$ \\
\midrule
0.5 & 3.536 & 3.512 & 3.360 & 3.244 & \textbf{3.217} \\
0.8 & 1.204 & 1.193 & 1.121 & 1.058 & \textbf{1.040} \\
1.0 & 0.752 & 0.745 & 0.700 & 0.663 & \textbf{0.653} \\
1.2 & 0.462 & 0.458 & 0.432 & 0.410 & \textbf{0.403} \\
1.6 & 0.237 & 0.235 & 0.221 & 0.208 & \textbf{0.205} \\
2.0 & 0.146 & 0.145 & 0.136 & 0.129 & \textbf{0.127} \\
\bottomrule
\end{tabular}
\end{table}

\begin{table}
\caption{MISE (multiplied by 1000) of the minimum $\phi$-divergence plug-in estimators of $qD$ for 1000 samples drawn from the Weibull distribution, assuming the Weibull distribution}
\label{tab:mise_weibull_weibull_qd}
\begin{tabular}{l|rrrrr}
\toprule
$\beta$ & HD & JSD & KLD & PD & $\chi^2_P$ \\
\midrule
0.5 & 3.286 & 3.265 & 3.125 & 3.020 & \textbf{2.995} \\
0.8 & 1.115 & 1.104 & 1.037 & 0.980 & \textbf{0.963} \\
1.0 & 0.702 & 0.699 & 0.659 & 0.627 & \textbf{0.617} \\
1.2 & 0.445 & 0.442 & 0.419 & 0.398 & \textbf{0.392} \\
1.6 & 0.242 & 0.240 & 0.225 & 0.213 & \textbf{0.209} \\
2.0 & 0.157 & 0.157 & 0.147 & 0.140 & \textbf{0.138} \\
\bottomrule
\end{tabular}
\end{table}

\begin{table}
\caption{MISE (multiplied by 1000) of the minimum $\phi$-divergence plug-in estimators of $qZ$ for 1000 samples drawn from the Dagum distribution, assuming the Dagum distribution}
\label{tab:mise_dagum_dagum_qz}
\begin{tabular}{c|rrrrr}
\toprule
$(a,r,b)$ & HD & JSD & KLD & PD & $\chi^2_P$ \\
\midrule
(2.8, 0.5, 8.0) & 1.722 & 1.683 & 1.559 & 1.466 & \textbf{1.440} \\
(2.2, 0.4, 9.0) & 1.293 & 1.286 & 1.243 & 1.212 & \textbf{1.203} \\
(4.6, 0.3, 11.0) & 2.046 & 1.919 & 1.681 & 1.545 & \textbf{1.538} \\
(4.0, 0.7, 4.0) & 1.889 & 1.855 & 1.742 & 1.656 & \textbf{1.632} \\
(6.0, 0.15, 15.0) & 1.644 & 1.635 & 1.520 & \textbf{1.434} & 1.456 \\
(3.2, 0.7, 4.0) & 2.098 & 2.081 & 2.018 & 1.960 & \textbf{1.939} \\
(3.7, 0.3, 13.0) & 1.758 & 1.726 & 1.589 & 1.511 & \textbf{1.494} \\
\bottomrule
\end{tabular}
\end{table}

\begin{table}
\caption{MISE (multiplied by 1000) of the minimum $\phi$-divergence plug-in estimators of $qD$ for 1000 samples drawn from the Dagum distribution, assuming the Dagum distribution}
\label{tab:mise_dagum_dagum_qd}
\begin{tabular}{c|rrrrr}
\toprule
$(a,r,b)$ & HD & JSD & KLD & PD & $\chi^2_P$ \\
\midrule 
(2.8, 0.5, 8.0) & 0.750 & 0.737 & 0.681 & 0.639 & \textbf{0.628} \\
(2.2, 0.4, 9.0) & 0.787 & 0.783 & 0.750 & 0.720 & \textbf{0.710} \\
(4.6, 0.3, 11.0) & 1.031 & 0.976 & 0.857 & 0.794 & \textbf{0.793} \\
(4.0, 0.7, 4.0) & 0.748 & 0.734 & 0.679 & 0.634 & \textbf{0.620} \\
(6.0, 0.15, 15.0) & 1.126 & 1.125 & 1.038 & 0.979 & \textbf{0.965} \\
(3.2, 0.7, 4.0) & 0.815 & 0.808 & 0.777 & 0.747 & \textbf{0.735} \\
(3.7, 0.3, 13.0) & 0.911 & 0.899 & 0.819 & 0.776 & \textbf{0.765} \\
\bottomrule
\end{tabular}
\end{table}

\subsection{Summary of the results}

Our simulation study comparing results obtained with multiple divergences shows that Pearson's chi-squared divergence leads to estimates with the lowest MSE in most cases, only slightly better than PD with $\lambda=\frac{2}{3}$.
The differences between the parameter estimators (and consequently between the inequality measure estimators), obtained with various divergences considered in this study, are very small.

\section{Real data analysis}\label{sec:real_data}
Two data sets were analysed to demonstrate the usefulness of the methods described above. 
The first shows possible applications in measuring income inequalities based on income data obtained from a~survey with limited options.
The second example is based on complete data, and its purpose is to compare the estimates of the indices based on these complete data with the estimates obtained after grouping.


\subsection{Income data from a~survey}
We analyse a~data set consisting of questionnaires filled out in the San Francisco Bay Area by customers of a~certain shopping mall (see, for example, Section 14.2.3 in \cite{Hastie2009}).
This data set is available in the \textit{IncomeESL} object of the \textit{arules} R package.
For each individual, the information about income is given by an interval. 
The edges of the income intervals are shown in Table \ref{tab:dagum_income_esl_probs}.
The survey also contains information on several demographic variables, including sex (\textit{sex}), age interval (\textit{age}), and education level (\textit{education}).

\begin{table}
\caption{Observed $P$ and fitted $G\left(\hat{\theta}^\phi_n\right)$ probabilities for the subsamples of \textit{incomeESL} data with various characteristics. The fitted probabilities were obtained by minimising $\chi^2_P$ statistic.}
\label{tab:dagum_income_esl_probs}
\begin{tabular}{c|lr|lr|lr|lr}
\toprule
& \multicolumn{4}{c|}{25-34} &  \multicolumn{4}{c}{35-44}  \\
& \multicolumn{2}{c|}{men} & \multicolumn{2}{c|}{women} & \multicolumn{2}{c|}{men} & \multicolumn{2}{c}{women} \\
income & $P$ & $G\left(\hat{\theta}^\phi_n\right)$ & $P$ & $G\left(\hat{\theta}^\phi_n\right)$ & $P$ & $G\left(\hat{\theta}^\phi_n\right)$ & $P$ & $G\left(\hat{\theta}^\phi_n\right)$ \\
\midrule
$[0,10)$ & 0.064 & 0.061 & 0.039 & 0.037 & 0.000 & 0.002 & 0.021 & 0.018 \\
$[10,15)$ & 0.064 & 0.063 & 0.050 & 0.055 & 0.006 & 0.010 & 0.021 & 0.025 \\
$[15,20)$ & 0.074 & 0.080 & 0.074 & 0.079 & 0.037 & 0.027 & 0.016 & 0.037 \\
$[20,25)$ & 0.080 & 0.091 & 0.110 & 0.095 & 0.049 & 0.053 & 0.053 & 0.049 \\
$[25,30)$ & 0.117 & 0.097 & 0.124 & 0.101 & 0.080 & 0.078 & 0.080 & 0.061 \\
$[30,40)$ & 0.177 & 0.185 & 0.142 & 0.188 & 0.178 & 0.196 & 0.165 & 0.151 \\
$[40,50)$ & 0.147 & 0.145 & 0.142 & 0.142 & 0.202 & 0.182 & 0.133 & 0.167 \\
$[50,75)$ & 0.181 & 0.182 & 0.202 & 0.181 & 0.258 & 0.263 & 0.330 & 0.310 \\
$75+$ & 0.097 & 0.096 & 0.117 & 0.122 & 0.190 & 0.189 & 0.181 & 0.181 \\
\midrule
p-value & \multicolumn{2}{c|}{0.583} & \multicolumn{2}{c|}{0.096} & \multicolumn{2}{c|}{0.629} & \multicolumn{2}{c}{0.144} \\
 $n$ & \multicolumn{2}{c|}{299} & \multicolumn{2}{c|}{282} & \multicolumn{2}{c|}{163} & \multicolumn{2}{c}{188} \\
\bottomrule
\end{tabular}
\end{table}

We analyse the level of income inequality among college graduates of both sexes from two age intervals 25--34 and 35--44.
The Dagum distribution was fitted to all four data samples considered using the $\chi^2_P$ statistic.
The comparison of the observed and expected probabilities in each bin can be seen in Table \ref{tab:dagum_income_esl_probs}.
At the bottom of Table \ref{tab:dagum_income_esl_probs}, we can also see the sizes of the samples, varying between 163 and 299, and the p-values of the chi-squared test, indicating that there is no evidence to reject the hypothesis of the goodness-of-fit to the Dagum distribution when the level of significance is 0.05.

\begin{figure}
    \centering
    \includegraphics[width=1.0\linewidth]{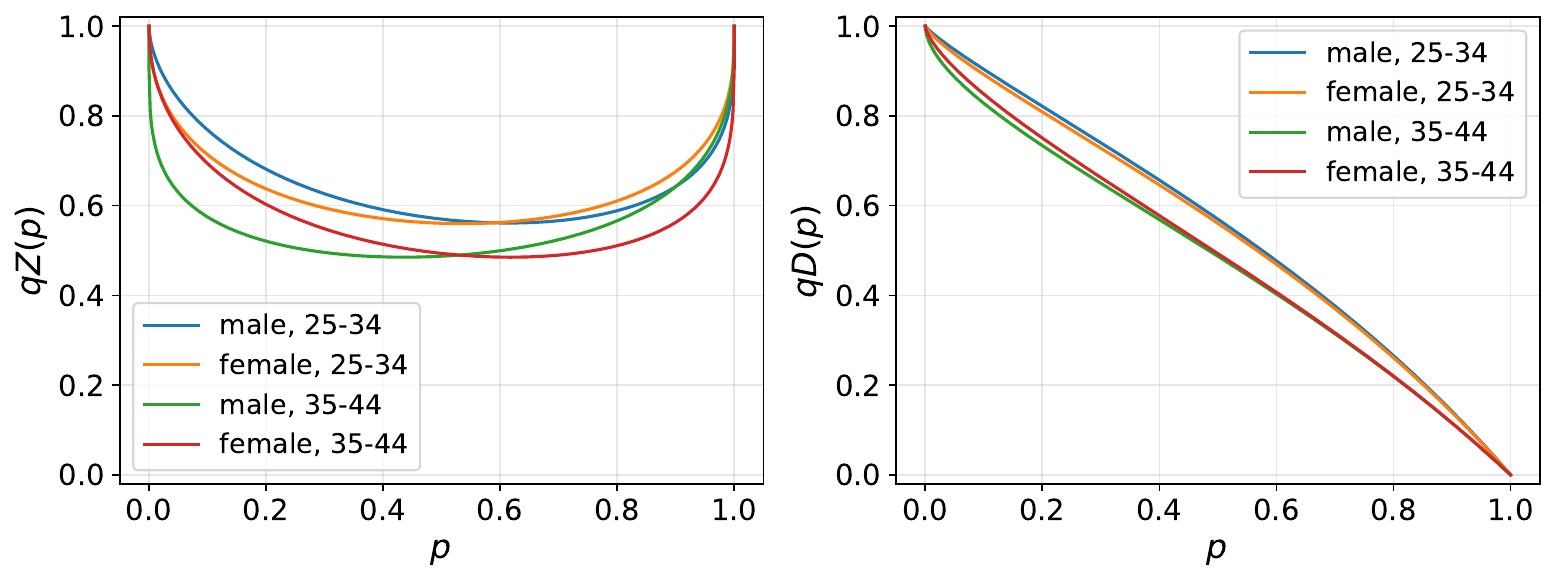}
    \caption{Estimated quantile inequality curves, $qZ$ (left panel) and $qD$ (right panel), by demographic group for \textit{IncomeESL} data set}
    \label{fig:dagum_qZ_qD_income_esl}
\end{figure}

Table \ref{tab:qzi_qdi_income_esl} shows the estimated values of \qziqdi in all four samples considered in this analysis. We can see that within a~given age group, the difference between the level of income inequality of men and women is rather small (approximately 0.01 for both \qziqdi).
However, for both men and women, the difference between the income inequality of younger (25--34) and older (35--44) individuals is significant, showing that the income inequality among older individuals is much lower. In the case of $qZI$ the difference is 0.089 for men and 0.066 for women, while in the case of $qDI$ it is 0.065 for men and 0.049 for women.
Both indices show that the income inequality is higher among men in the group of individuals of age 25--34, while among those of age 35--44 the income inequality is higher among women.
Figure \ref{fig:dagum_qZ_qD_income_esl} shows the estimated curves $qZ$ (left panel) and $qD$ (right panel) for these four groups.
Interestingly, in both age groups, the $qD$ curves are almost identical for men and women.
However, $qZ$ curves differ significantly between men and women, showing that $qZ$ describes different aspects of income inequality than $qD$.
This example shows that $qZ$ curves corresponding to two different distributions can differ visibly even if the values of the indices are very close.

\begin{table}[htbp]
    \caption{Estimated quantile inequality measures by demographic group for \textit{IncomeESL} data set}
    \label{tab:qzi_qdi_income_esl}
    \centering
    \begin{tabular}{lccc}
        \toprule
        Age & Sex & $q{ZI}$ & $q{DI}$ \\
        \midrule
        25--34 & male   & 0.644 & 0.545 \\
        25--34 & female & 0.634 & 0.537 \\
        35--44 & male   & 0.555 & 0.480 \\
        35--44 & female & 0.568 & 0.488 \\
        \bottomrule
    \end{tabular}
\end{table}

\subsection{Current Population Survey}

We analyse a~data set of the Current Population Survey (CPS) \cite{fedesoriano} that contains information about salaries, gathered in the USA between 1981 and 2013.
Since the data set is huge, we narrow the scope of this analysis to the salaries of male production workers in 1990.
The income inequality within these groups was then analysed.
We grouped 589 observations for men aged 25--29 and 660 observations for men aged 30--34 into 10 bins with edges given in Table \ref{tab:men_wage_fitted_weibull}.
Since complete samples are available, we group them ourselves and check whether estimation from the grouped data yields values close to those obtained from the complete data.
The $p$-values of the $\chi^2$ (with 7 degrees of freedom) goodness-of-fit test checking if the Weibull distribution fits the grouped data were 0.061 and 0.083, respectively, for younger and older employees, so we do not have evidence to reject the null hypothesis about the goodness-of-fit to the Weibull distribution on the level of significance 0.05.

\begin{table}[H]
\caption{Observed and fitted probabilities of falling into a~given wage group by men aged 25--29 and 30--34 working as production workers. Expected probabilities were obtained by fitting the Weibull distribution minimising the $\chi^2_P$ statistic.}
\label{tab:men_wage_fitted_weibull}
\begin{tabular}{c|rr|rr}
\toprule
age group & \multicolumn{2}{c|}{25--29} & \multicolumn{2}{c}{30--34} \\
yearly wage bins & & & & \\
(in thousands of USD)& observed & expected & observed & expected \\
\midrule
0 -- 5 & 0.0374 & 0.0332 & 0.0242 & 0.0198 \\
5 -- 10 & 0.0849 & 0.1069 & 0.0636 & 0.0738 \\
10 -- 15 & 0.1715 & 0.1640 & 0.1152 & 0.1271 \\
15 -- 20 & 0.2003 & 0.1869 & 0.1803 & 0.1623 \\
20 -- 25 & 0.1868 & 0.1741 & 0.1727 & 0.1705 \\
25 -- 30 & 0.1307 & 0.1375 & 0.1561 & 0.1530 \\
30 -- 35 & 0.0849 & 0.0935 & 0.1227 & 0.1192 \\
35 -- 40 & 0.0526 & 0.0552 & 0.0742 & 0.0812 \\
40 -- 50 & 0.0407 & 0.0412 & 0.0682 & 0.0742 \\
50 -- $\infty$ & 0.0102 & 0.0075 & 0.0227 & 0.0190 \\
\bottomrule
\end{tabular}
\end{table}

We obtained the estimated values of the parameters of the Weibull distribution by minimising the $\chi^2_P$ statistic.
The estimated values of the Weibull distribution parameters and consequently the estimated values of the inequality indices can be seen in Table \ref{tab:men_wage_indices}.
We can see in Table \ref{tab:men_wage_indices} that the values estimated based on grouped data are very close to the values estimated for the complete samples.
Table \ref{tab:men_wage_indices} contains also information about the confidence intervals of the estimated values of the indices, based on the normal approximation obtained in Theorem \ref{thm:qzi_qdi_as_norm}.
Figure \ref{fig:weibull_qZ_qD_cps} shows the estimated curves $qZ$ (left panel) and $qD$ (right panel) for both age groups.
We can see that both in the case of $qZ$ and $qD$, the curves for both age groups are very similar, which is a~direct consequence of the fact that the difference between the estimated values of the shape parameter is small (0.135).

\begin{table}[H]
    \caption{Estimated values of the parameters of the Weibull distribution based on complete sample (ML estimator) and based on frequencies of the grouped data ($\chi^2_P$ estimator) and the values of the plug-in estimators of $qZI$ and $qDI$ in the case of CPS data}
    \label{tab:men_wage_indices}
    \begin{tabular}{c|cc|cc}
    \toprule
     & \multicolumn{2}{c|}{25--29} & \multicolumn{2}{c}{30--34} \\
         & ML & $\chi^2_P$ & ML & $\chi^2_P$ \\
    \midrule
    $\left(\hat{\beta}, \hat{\sigma}\right)$ & (2.124, 23.312) & (2.162, 23.982) & (2.256, 26.940) & (2.297, 27.459) \\
    $qZI$ & 0.581 & 0.575 & 0.560 &  0.553\\
    $qZI$ CI & (0.580, 0.583) & (0.573 0.576) & (0.559, 0.562) & (0.552 0.555) \\
    $qDI$ & 0.506 & 0.502 & 0.490 &  0.485\\
    $qDI$ CI & (0.506, 0.508) & (0.500 0.503) & (0.490, 0.492) & (0.484 0.486) \\
    \bottomrule
    \end{tabular}
\end{table}

\begin{figure}
    \centering
    \includegraphics[width=1.0\linewidth]{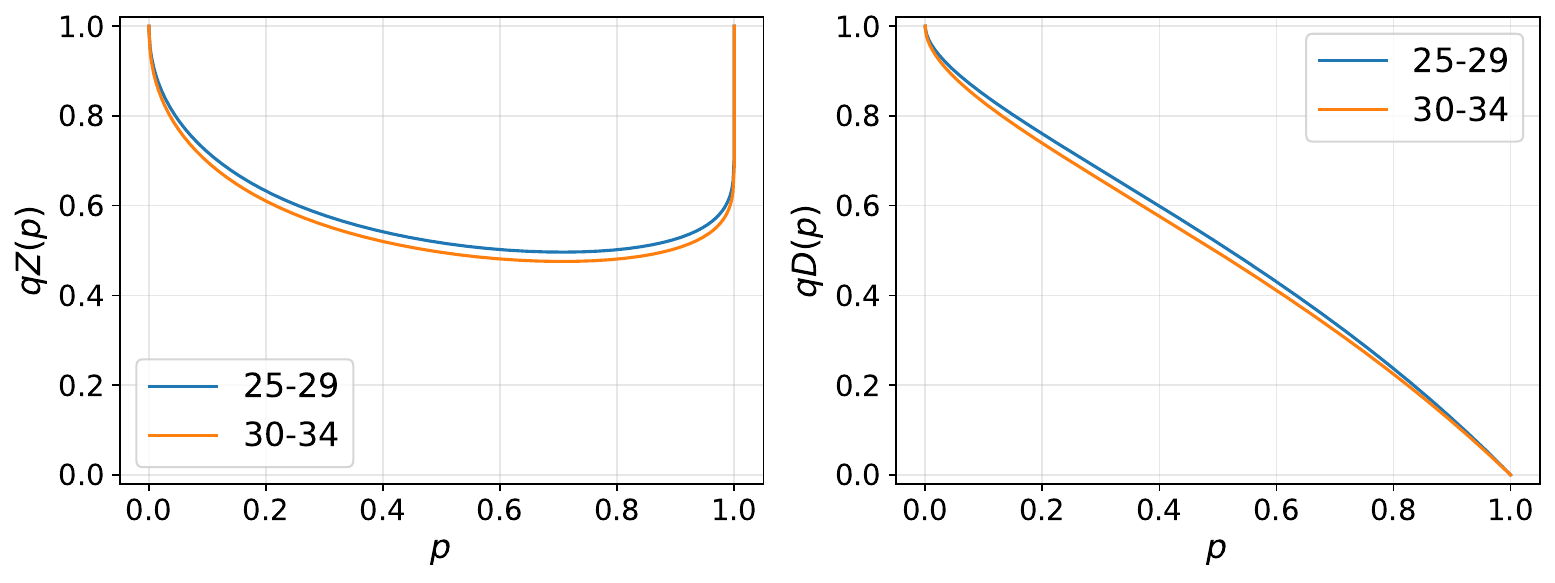}
    \caption{Estimated quantile inequality curves, $qZ$ (left panel) and $qD$ (right panel), by age group for \textit{CPS} data set}
    \label{fig:weibull_qZ_qD_cps}
\end{figure}

\section{Conclusions}\label{sec:conclusions}

This paper explores the problem of estimating the inequality indices based on grouped data.
A parametric approach with various minimum $\phi$-divergences was applied to estimate the parameters $\theta$ of the assumed distribution of the incomes.
The sufficient assumptions for such a~minimum $\phi$-divergence estimator of $\theta$ to be the best asymptotically normal (BAN) were described.
The plug-in estimators of the quantile inequality curves $qZ$ and $qD$ were shown to be consistent, while the plug-in estimators of the measures $qZI$ and $qDI$ were shown to be asymptotically normal.
A simulation study shows that the power divergence with $\lambda=\frac{2}{3}$ and Pearson's chi-squared divergence lead to the most stable estimators 
that attain nearly minimal MSE, with $\chi^2_P$ marginally ahead.

\subsection*{Data and Code Availability}\label{sec:data_and_code}
The code used to reproduce the simulation study, the empirical analysis, and the figures presented in this article are available on a~dedicated GitHub repository:
\url{https://github.com/skfp/grouped_data}.


\vspace{1cm}
\noindent
\textbf{Acknowledgement.} S.P. acknowledges the support of Dioscuri programme initiated by the Max Planck Society, jointly managed with the National Science Centre (Poland), and mutually funded by the Polish Ministry of Science and Higher Education and the German Federal Ministry of Research, Technology and Space.\\
S.P. was supported by the National Science Centre (NCN), Poland, under the Preludium funding scheme, grant number 2025/57/N/HS4/02684.

\bibliographystyle{acm}

\end{document}